\documentclass[12pt]{amsart} 
\usepackage[T1]{fontenc}
\usepackage[utf8]{inputenc}
\usepackage{lmodern}
\usepackage{microtype} 

\usepackage{mathtools} 
\usepackage{tikz}

\newtheorem{Theorem}{Theorem}
\newtheorem{Corollary}[Theorem]{Corollary} 
\newtheorem{Lemma}[Theorem]{Lemma}

\newcommand{\rr}{\mathbb{R}}

\begin{document}

\title{Growth series of CAT(0) cubical complexes}

\author{Boris Okun} \address{University of Wisconsin--Milwaukee}
\email{okun@uwm.edu}
\author{Richard Scott}\address{Santa Clara University}
\email{rscott@scu.edu} \thanks{Both authors partially supported by a Simon's Foundation Collaboration Grant for Mathematicians.}

\date{\today}
\begin{abstract}
	Let $X$ be a CAT(0) cubical complex.
	The growth series of $X$ at $x$ is $G_{x}(t)=\sum_{y \in Vert(X)} t^{d(x,y)}$, where $d(x,y)$ denotes $\ell_{1}$-distance between $x$ and $y$.
	If $X$ is cocompact, then $G_{x}$ is a rational function of $t$.
	In the case when $X$ is the Davis complex of a right-angled Coxeter group it is a well-known that $G_{x}(t)=1/f_{L}(-t/(1+t))$, where $f_{L}$ denotes the $f$-polynomial of the link $L$ of a vertex of $X$.
	We obtain a similar formula for general cocompact $X$.
	We also obtain a simple relation between the growth series of individual orbits and the $f$-polynomials of various links.
	In particular, we get a simple proof of reciprocity of these series ($G_{x}(t)=\pm G_{x}(t^{-1})$) for an Eulerian manifold $X$.
\end{abstract}
\maketitle

Let $X$ be a CAT(0) cube complex with a cocompact cellular action by a group $G$.
Denote by $d(x,y)$ the $\ell_{1}$-distance between vertices $x$ and $y$ of $X$.
We consider the following growth series:
\[
	G_{xy}=\sum_{z \in Gy} t^{d(x,z)}
\]
--- the growth series of $G$-orbit of $y$ as seen from $x$, and
\[
	G_{x}=\sum_{y \in X} t^{d(x,y)}
\]
--- the full growth series of $X$ as seen from $x$.

The aim of this paper is to establish relations between these growth series and the local structure of $X$ and $X/G$.
In order to do this we introduce more notation.
The $f$-polynomial of a simplicial complex $L$ is given by:
\[
	f_{L}(t)=\sum_{\sigma \in L} t^{\dim \sigma +1}.
\]
Note that we assume that $L$ contains an empty simplex of dimension $-1$, so the $f$-polynomial always has free term $1$.
For vertices $x$ and $y$ of $X$, denote by $\Box_{xy}$ the cube spanned by $x$ and $y$.
In other words $\Box_{xy}$ is the minimal cube containing $x$ and $y$.
Let $f_{xy}$ denote the $f$-polynomial of the link of the cube $\Box_{xy}$, and let $f_{x}=f_{xx}$ denote the $f$-polynomial of the link of the vertex $x$.
We put $f_{xy}=0$ if $x$ and $y$ are not contained in a cube.

A fundamental example of a cocompact CAT(0) cube complex is the Davis complex of a right-angled Coxeter group.
In this case the group acts simply transitively on vertices and thus all the growth series are equal, and we have the following well-known result.
(For general Coxeter groups, this can be found in \cite{s68}, Theorem 1.25 and Corollary 1.29. For the right angled case, it takes the following form.)
\begin{Theorem}
	If $G$ is a right-angled Coxeter group and $X$ is its Davis complex, then
	\[
		G_{x}(t) f_{x}\left(\frac{-t}{1+t}\right) =1.
	\]
\end{Theorem}

In fact, it was proved by the second author in \cite{s07} that the same formula holds if one assumes only that the $f$-polynomials of all vertices are the same:
\begin{Theorem}
	If the links of all vertices of $X$ have the same $f$-polynomial, then
	\[
		G_{x}(t) f_{x}\left(\frac{-t}{1+t}\right) =1.
	\]
\end{Theorem}

Our goal is to generalize this to the case of different links.
Since, by a result of Niblo and Reeves \cite{nr98}, CAT(0) cube groups have an automatic structure, it follows that the growth series $G_{xy}$ are rational functions of $t$ computable in terms of the local structure of $X$.
This computation was carried out by the second author in \cite{s14b}, where it was used to prove reciprocity of the growth series for Eulerian manifolds.
In this paper we obtain different and much simpler formulas for the growth series which lead to an easy proof of reciprocity.

Our result is easiest to state when the action is sufficiently free.
Define:
\begin{equation}\label{e:cxy}
	c_{xy}=\left(\frac{-t}{1-t^{2}}\right)^{d(x,y)} f_{xy}\left(\frac{t^{2}}{1-t^{2}}\right).
\end{equation}
\begin{Theorem}\label{t:simple}
	If the stars of vertices are embedded in $X/G$, then the matrices $(G_{xy})$ and $(c_{xy})$ ($x,y \in X/G$) are inverses of each other.
\end{Theorem}
Note that in this case the matrices are symmetric.

In the general case, two vertices in $X/G$ can span multiple cubes.
To account for this we modify the coefficients $c_{xy}$ as follows.
Let $\pi: X \to X/G$ denote the natural projection.
For $x,y \in X/G$, pick $\bar{x} \in \pi^{-1}(x)$ and set
\[
	\bar{c}_{xy}= \sum_{\bar{y} \in \pi^{-1}(y)} c_{\bar{x}\bar{y}}.
\]
\begin{Theorem}\label{t:general}
	The matrices $(G_{xy})$ and $(\bar{c}_{xy})$ are inverses of each other:
	\[
		\sum_{y \in X/G} \bar{c}_{xy} G_{yz} = \delta_{xz}.
	\]
\end{Theorem}
\begin{Corollary}\label{c:recip}
	If $X$ is an $n$-dimensional Eulerian manifold, then $c_{xy}$ and $G_{xy}$ satisfy reciprocity:
	\begin{align*}
		c_{xy}(t^{-1}) &=(-1)^{n} c_{xy}(t), \\
		G_{xy}(t^{-1}) &=(-1)^{n} G_{xy}(t).
	\end{align*}
\end{Corollary}
\begin{proof}
	For a simplicial Eulerian $(n-1)$-sphere $L$ we have the Dehn--Sommerville relations (see \cite{s97b}, pages 353-354 or \cite{g05}, p.271)
	\[
		f_{L}(t - 1) = (-1)^{n}f_{L}(-t).
	\]
	
	A bit of algebra gives the first formula, and the second formula then follows.
\end{proof}

At this point an attentive reader might wonder how to reconcile our formula with the one for the Davis complex, where our matrices become $1\times1$.
We have the following lemma:
\begin{Lemma}\label{c:sum1}
	\[
		\sum_{y \in X} c_{xy}=f_{x}(\frac{-t}{1+t}).
	\]
	\[
		\sum_{y \in X/G} \bar{c}_{xy}=f_{x}(\frac{-t}{1+t}).
	\]
\end{Lemma}
\begin{proof}
	The first statement is true for an $n$-cube as both sides evaluate to $\frac{1}{(1+t)^{n}}$, and both sides behave the same under taking unions.
	The second statement follows from the first.
\end{proof}
Summing the main formula $\sum_{y \in X/G} \bar{c}_{xy} G_{yz} = \delta_{xz} $ over $x$, or $z$, or both, and using the previous Lemma gives:
\begin{Corollary}
	\[
		\sum_{y \in X/G} f_{y}\left( \frac{-t}{1+t} \right) G_{xy} = 1,
	\]
	\[
		\sum_{y \in X/G} \bar{c}_{xy}G_{y}=1,
	\]
	\[
		\sum_{x \in X/G} f_{x}\left( \frac{-t}{1+t} \right) G_{x}=\#X/G.
	\]
\end{Corollary}
So we indeed recover the Davis complex formula.

Our proof of Theorems~\ref{t:simple} and \ref{t:general} is based on a different description of the entries of the inverse of the matrix $G_{xy}$.
We develop this description in the next four lemmas before proving the Theorems.
For each vertex $x \in X$ define a function $h_{x}:X \to \rr[t]$ by $h_{x}(y)=t^{d(x,y)}$.
The following lemma is key in our approach.
\begin{Lemma}\label{l:key}
	Let $x \in X$, and let $S=Vert(St(x))$ denote the vertices of the cubical star of $x$.
	Then the characteristic function of $\{x\}$, $1_{x}$ is a unique linear combination of the functions $h_{y}$, $y\in S$, over $\rr(t)$, the field of rational functions.
\end{Lemma}
\begin{proof}
	Since $X$ is CAT(0), the hyperplanes near $x$ (corresponding to edges starting at $x$)
	divide $X$ into convex polyhedral regions.
	Each region $R$ has a unique vertex $r$ closest to $x$.
	Also, $r \in S$.
	We will refer to the regions as cones and to the vertices as cone points.
	
	For any $z \in R$ and $y\in S$ there exist a geodesic edge path which goes through $r$.
	\[
		d(y,z)=d(y,r) + d(r,z).
	\]
	Therefore,
	\[
		h_{y}(z)=t^{d(r,z)} h_{y}(r).
	\]
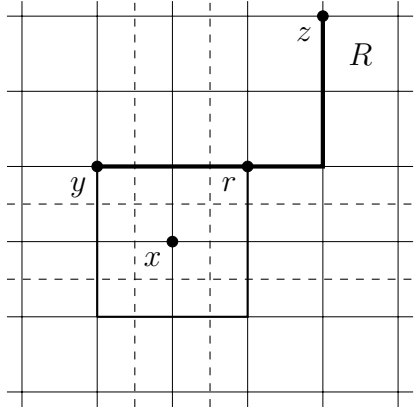
	\begin{center}
		\begin{tikzpicture}[scale=1]
			\def\n{3.2}
			\def\m{-\n+1}
			
			\draw[ultra thin] (\m, \m) grid (\n, \n);
			\draw[dashed] (\m, .5) -- (\n, .5);
			\draw[dashed] (\m, -.5) -- (\n, -.5);
			\draw[dashed] (.5, \m) -- (.5, \n);
			\draw[dashed] (-.5, \m) -- (-.5, \n);
			\coordinate (z) at (2,3) ;
			\coordinate (y) at (-1,1);
			\coordinate (r) at (1,1);
			\coordinate (s) at (-1,-1);
			\draw[thick] (s)-|(r);
			\draw[thick] (s)|-(r);
			\node at (2.5, 2.5) {$R$};			
			\filldraw (0,0) circle (2pt)[fill=black] node[below left] {$x$}; 
			\filldraw (z) circle (2pt)[fill=black] node[below left] {$z$}; 
			\filldraw (y) circle (2pt)[fill=black] node[below left] {$y$}; 
			\filldraw (r) circle (2pt)[fill=black] node[below left] {$r$};
			\draw[ultra thick] (y)--(r);
			\draw[ultra thick] (r)-|(z);
		\end{tikzpicture}
	\end{center}
	
	This implies that for a fixed cone the values of all the functions $h_{y}(z)$ are the same power of $t$ multiple of the corresponding value at the cone point.
	It follows that if a linear combination of $h_{y}$, $y\in S$ vanishes at a cone point, then it vanishes on the whole cone.
	
	Thus, since the cone points are precisely $S$ it is enough to prove the special case, when $X=St(x)$.
	
	In this case the $S \times S$ matrix of values of $h$-functions $(h_{y}(z))=(t^{d(y,z)})$ has $1$'s on the diagonal and positive powers of $t$ off the diagonal.
	Its determinant is a nonzero polynomial, since it evaluates to $1$ at $t=0$.
	Therefore, the matrix is invertible over $\rr(t)$ and the desired coefficients are given by the $x$-row of its inverse.
\end{proof}
\begin{Lemma}
	If $X=A\times B$ then the coefficients for $X$ are products of coefficients for $A$ and $B$.
\end{Lemma}
\begin{proof}
	This is immediate from the formula
	\[
		h_{(a,b)}(x,y)=h_{a}(x)h_{b}(y).
	\]
\end{proof}
One of the implications of the proof of Lemma~\ref{l:key} is that it is enough to understand the case when $X=St(x)$.
Note that when $X$ is a star, we do not need to assume that $X$ is CAT(0).

So assume that $X=St(x)$ and denote the coefficients posited in the Lemma by $c_{xy}^{X}$.
This should not cause confusion since we will show in Lemma~\ref{l:expl} that they are same as $c^{}_{xy}$ given by (\ref{e:cxy}).
Our proof is based on building $X$ inductively cube by cube and using the product formula and a certain inclusion--exclusion formula (Lemma~\ref{l:ie} below.)

In order to state the inclusion--exclusion formula we introduce more notation.
If $A$ is a sub-complex of $X$ containing $x$, which is also a star $A=St_{A}(x)$, then we extend the coefficients $c_{xy}^{A}$ to all of $X$ by setting $c_{xy}^{A}=0$ for $y \not\in A$.
Since for $y\in A$ the function $h_{y}$ for $X$ restricts to the function for $A$, we have
\[
	\sum_{y \in X} c_{xy}^{A} h_{y}(z) = \sum_{y \in A} c_{xy}^{A} h_{y}(z),
\]
and the resulting function restricts to $1_{x}$ on $A$.

The basis of our induction is the following.
If $X=\{x\}$ then the only coefficient is 1. For a segment $X=[xy]$ the coefficients are $c_{xx}^{X}=1/(1-t^{2})$ and $c_{xy}^{X}=-t/(1-t^{2})$:
\[
	1_{x}=\frac{1}{1-t^{2}} h_{x} + \frac{-t}{1-t^{2}} h_{y}.
\]
\begin{Lemma}[Inclusion--exclusion]\label{l:ie}
	If $X=St(x)$ decomposes as $X=A\cup_{C} B$, where $A$, $B$ and $C$ are subcomplexes of $X$ which are also stars of $x$, then
	\[
		c_{xy}^{X} = c_{xy}^{A} +c_{xy}^{B} -c_{xy}^{C}.
	\]
\end{Lemma}
\begin{proof}
	
	First consider a special case when $A=C\times[xz]$ is the star of the edge $[xz]$.
	We identify $C$ with $C\times\{x\}$.
	Then, for $a=(c,z) \in A-C$ and $b \in B$ we can choose a geodesic through $c=(c,x)\in C$ and therefore
	\[
		h_{a}(b)=th_{c}(b) \qquad h_{b}(a)=t h_{b}(c).
	\]
	Also, from the product formula we have
	\[
		c_{xc}^{A}= \frac{1}{1-t^{2}} c_{xc}^{C} \qquad c_{xa}^{A}= \frac{-t}{1-t^{2}} c_{xc}^{C}.
	\]
	It follows that
	\begin{multline*}
		\sum_{y \in A} c_{xy}^{A}h_{y}(b) = \sum_{a \in A-C} c_{xy}^{A}h_{a}(b) + \sum_{c \in C} c_{xy}^{A} h_{c}(b) \\
		=\sum_{c \in C} \frac{-t}{1-t^{2}} c_{xc}^{C}t h_{c}(b)+ \sum_{c \in C} \frac{1}{1-t^{2}}c_{xc}^{C} h_{c}(b)= \sum_{c \in C} c_{xc}^{C} h_{c}(b),
	\end{multline*}
	\[
		\sum_{b \in B} c_{xb}^{B} h_{y}(a) = t \sum_{b \in B} c_{xb}^{B}h_{b}(c) = t 1_{x}(c).
	\]
	and
	\[
		\sum_{c \in C} c_{xb}^{C} h_{b}(a) = t \sum_{b \in C} c_{xb}^{C}h_{b}(c) = t 1_{x}(c).
	\]
	
	Therefore,
	\begin{multline*}
		\sum_{y \in X} (c_{xy}^{A} +c_{xy}^{B} -c_{xy}^{C} ) h_{y}(b) = \left[ \sum_{y \in A} c_{xy}^{A} h_{y}(b) - \sum_{ c \in C} c_{xc}^{C} h_{c}(b) \right] + \sum_{y \in B} c_{xy}^{B} h_{y}(b) \\
		=1_{x}(b),
	\end{multline*}
	and
	\begin{multline*}
		\sum_{y \in X} (c_{xy}^{A} +c_{xy}^{B} -c_{xy}^{C} ) h_{y}(a) = \sum_{y \in A} c_{xy}^{A} h_{y}(a) - \left[ \sum_{ c \in C} c_{xc}^{C} h_{c}(a) - \sum_{b \in B} c_{xb}^{B} h_{b}(a) \right] \\
		= 1_{x}(a)=0,
	\end{multline*}
	since the bracketed differences vanish.
	
	Thus $\sum_{y \in X} (c_{xy}^{A} +c_{xy}^{B} -c_{xy}^{C} ) h_{y}=1_{x}$, and the special case follows from uniqueness of the coefficients.
	
	The general case now follows by induction.
	Let $X=A\cup_{C} B$ and let $[xz]$ be an edge in $X$.
	As before, associated to the edge we have decomposition of $X$ into the star of $[xz]$ and the rest, which we write as $X=X_{1} \cup_{X_{3}} X_{2}$, for which the inclusion-exclusion formula holds.
	\begin{equation}\label{eq:x}
		c_{xy}^{X} = c_{xy}^{X_{1}} +c_{xy}^{X_{2}} -c_{xy}^{X_{3}}.
	\end{equation}
	Intersecting this decomposition with the original one gives decompositions $A=A_{1} \cup_{A_{3}} A_{2}$, and similarly of $B$ and $C$.
	It also gives decompositions of $X_{i}=A_{i}\cup_{C_{i}} B_{i}$.
	The left hand sides of these 6 decompositions are proper subsets of $X$ and we can assume by induction that the inclusion--exclusion formula holds for them.
	\[
		\begin{aligned}
			c_{xy}^{X_{1}} &= c_{xy}^{A_{1}} +c_{xy}^{B_{1}} -c_{xy}^{C_{1}} \\
			c_{xy}^{X_{2}} &= c_{xy}^{A_{2}} +c_{xy}^{B_{2}} -c_{xy}^{C_{2}} \\
			c_{xy}^{X_{3}} &= c_{xy}^{A_{3}} +c_{xy}^{B_{3}} -c_{xy}^{C_{3}} \\
		\end{aligned}
		\qquad
		\begin{aligned}
			c_{xy}^{A} &= c_{xy}^{A_{1}} +c_{xy}^{A_{2}} -c_{xy}^{A_{3}} \\
			c_{xy}^{B} &= c_{xy}^{B_{1}} +c_{xy}^{B_{2}} -c_{xy}^{B_{3}} \\
			c_{xy}^{C} &= c_{xy}^{C_{1}} +c_{xy}^{C_{2}} -c_{xy}^{C_{3}}
		\end{aligned}
	\]
	Substituting the formulas in the first column into (\ref{eq:x}) and comparing with $c_{xy}^{A} +c_{xy}^{B} -c_{xy}^{C}$ using the second column verifies the desired formula for $X=A\cup_{C} B$.
\end{proof}

Below are some examples.
\begin{center}
	\begin{tikzpicture}[scale=1.5]
		
		\def\d{1-t^{2}}
		\draw (0,0) node[left] {$x$} node[above] {$\frac{1}{\d} $} -- (1,0) node[above] {$\frac{-t}{\d} $} ;
		\begin{scope}[xshift=3cm]
			\def\d{(1-t^{2})^{2}}
			\draw (0,0) node[left] {$x$} node[below] {$\frac{1}{\d} $} -- (1,0) node[below] {$\frac{-t}{\d} $} -- (1,1) node[above] {$\frac{t^{2}}{\d} $} -- (0,1) node[above] {$\frac{-t}{\d} $} -- cycle ;
		\end{scope}
		\begin{scope}[yshift=-3cm, xshift=2cm]
			\def\d{(1-t^{2})^{2}}
			\draw (0,0) node[above left] {$x$} node[below] {$\frac{2}{\d} - \frac{1}{1-t^{2}}$} -- (1,0) node[ right] {$\frac{-t}{\d} $} -- (1,1) node[ right] {$\frac{t^{2}}{\d} $} -- (0,1) node[above] {$\frac{-2t}{\d} - \frac{-t}{1-t^{2}} $} -- cycle ;
			\draw (0,0) -- (-1,0) node[ left] {$\frac{-t}{\d} $} -- (-1,1) node[ left] {$\frac{t^{2}}{\d} $} -- (0,1) ;
		\end{scope}
	\end{tikzpicture}
\end{center}

In fact, we have explicit formulas for the coefficients:
\begin{Lemma}\label{l:expl}
	The coefficient of $h_{y}$ is precisely the $c_{xy}$ introduced before:
	\[
		c_{xy}= \left(\frac{-t}{1-t^{2}}\right)^{d(x,y)} f_{xy}\left(\frac{t^{2}}{1-t^{2}}\right).
	\]
\end{Lemma}
\begin{proof}
	By the product formula this is true if $X$ is a cube, and both sides behave the same under taking unions.
\end{proof}

We are now in position to finish the proof of the Theorems.
\begin{proof}
	Since $c_{xy}=0$ for $x$ and $y$ not spanning a cube, we have:
	\[
		\sum_{y \in X} c_{xy} h_{y} = \sum_{y \in St(x)} c_{xy} h_{y} = 1_{x}.
	\]
	Since
	\[
		G_{yz}=\sum_{w\in Gz}h_{y}(w),
	\]
	we have
	\[
		\sum_{y \in X} c_{xy} G_{yz} = \sum_{\substack{y \in X \\
		w\in Gz }} c_{xy} h_{y}(w) = \sum_{ w\in Gz } 1_{x}(w)= \delta_{\pi(x)\pi(z)}.
	\]
	
	Since $G_{xy}$ are $G$-invariant in both variables, and $c_{xy}$ is invariant under the diagonal action, we can express this result in terms of $X/G$, to obtain Theorems \ref{t:simple} and \ref{t:general}.
\end{proof}
Finally, we note that $c_{xx}=f_{x} \left( \frac{t^{2}}{1-t^{2}}\right) $.
Taking $t=\sqrt{-1}$ we obtain the following strange corollary:
\begin{Corollary}\label{c:inverse}
	If the stars of vertices are embedded in $X/G$, then
	\[
		tr(c_{xy}(\sqrt{-1}))=\chi(X/G).
	\]
	Thus the Euler characteristics can be computed from the matrix of the growth series $(G_{xy})$ evaluated at $\sqrt{-1}$.
\end{Corollary}

\end{document}